\documentclass{amsart}
\usepackage{HMStyle}
\usepackage{changepage}

\begin{document}

\title[Relative Du Bois complexes]{Relative Du Bois complexes over higher dimensional bases}

\author[D.Kassabova]{Dora Kassabova}\address{University of Washington, Department of Mathematics, Seattle, WA 98195, USA}
\email{dmk285@uw.edu}

\author[H.Ning]{Haoming Ning}
\address{School of Mathematics, Georgia Institute of Technology, GA 30332, USA}
\email{hning31@gatech.edu}

\date{\today}
\subjclass[2020]{14B05, 14J17, 14F10}

\begin{abstract}
We extend the notion of the right relative Du Bois complex to a smooth base of arbitrary dimension. We further show that this complex satisfies base change in the sense of Kov\'acs--Taji, for flat families with strict higher Du Bois fibers, extending the result by Ning over a smooth curve.
\end{abstract}

\maketitle


\medskip
\section{Introduction}

For a smooth morphism $f:X \to S$ of smooth complex varieties, the \emph{relative de Rham complex} $\Omega^\kdot_{X/S}$ is characterized by the exact sequence:
$$
0 \xrightarrow[]{} f^*\Omega^1_S \xrightarrow[]{} \Omega^1_X \xrightarrow[]{} \Omega^1_{X/S} \xrightarrow[]{} 0.
$$
Exterior powers of this sequence give rise to a natural filtration of $\Omega^p_X$ with graded pieces $f^*\Omega_S^q\otimes\Omega^{p-q}_{X/S}$. A key property is that $\Omega^p_{X/S}$ satisfies base change, that is, for $f$ smooth,
$$\Omega^p_{X/S} \otimes^{\tL}_X\scrO_{X_s} \qis \Omega^p_{X_s}.$$

For $X$ singular, Du Bois \cite{DB81} replaced the de Rham complex $\Omega_X^\kdot$ with the filtered complex $\uOmega_X^\kdot$ in the derived category of $X$. It agrees with $\Omega_X^\kdot$ when $X$ is smooth, and, from a cohomological point of view, this complex behaves very much like the usual de Rham complex. It is therefore natural to ask whether there is a way to produce an object analogous to the relative de Rham complex of a smooth family, equipped with a similar base change property.

Kov\'acs constructed a (left) relative Du Bois complex, which we will denote $\uOmega^{p,-}_{X/S}$, over a smooth curve in \cite{Kov96} and over a smooth base of arbitrary dimension in \cite{Kov97a,Kov97}. It shows many applications to families of singular varieties; see \cite{Kov96, Kov97, KT23, KT24}. One important question of Kov\'acs--Taji \cite{KT25} asks whether the relative Du Bois complex satisfies base change, that is, does
\[
\uOmega^{p,-}_{X/S} \otimes^{\tL}_X\scrO_{X_s} \qis \uOmega^p_{X_s}
\]
hold? It turns out that, in general, base change for the left relative Du Bois complex fails in every degree \cite{JK25, Nin26}. Ning proposed an alternative construction, dubbed the right relative Du Bois complex $\uOmega^{p,+}_{X/C}$ for $f: X \to C$, where $C$ is a smooth curve. It satisfies base change \cite[Theorem 1.1]{Nin26} for flat families with strict higher Du Bois singularities (cf. Definition \ref{def_higher_DB}). As shown in \emph{loc. cit.}, base change properties have geometric applications such as deformation invariance of higher Du Bois singularities.

The purpose of this paper is to extend the notion of the right relative Du Bois complex $\uOmega^{p,+}_{X/S}$ to a smooth base of arbitrary dimension while preserving base change properties, see Definition \ref{def_bc} for the precise definition. We construct the following.

\begin{theorem} (= Theorem \ref{thm_relDB})
Let $f: X\to S$ be a morphism of complex algebraic varieties with $\dim X=n$, $\dim S=k$ and $S$ smooth. Then there exist objects $\uOmega^{p, +}_{X/S}, \uOmega^{p, -}_{X/S}$ in $D^b_{\rm coh}(X)$ for all $p\in \bbZ$ with the following properties:
\begin{enumerate}
\item $\uOmega^{p,+}_{X/S}\qis 0$ for $p< 0$, $\uOmega^{p, -}_{X/S}\qis 0$ for $p>n-k$.

\item There exist hyperfiltrations $\bbF^{q,+}\uOmega^p_X, \bbF^{q,-}\uOmega^p_X$ of $\uOmega^p_X$ with $0\le q\le k+1$, such that
\[
\bbF^{0,+}\uOmega^p_X \qis \bbF^{0,-}\uOmega^p_X \qis \uOmega^p_X, \quad \bbF^{k+1,+}\uOmega^p_X \qis \bbF^{k+1,-}\uOmega^p_X \qis 0,
\]
whose associated graded complexes, with $0\le q\le k$ are,
\[
\bbG^{q,+}\uOmega^p_X\qis \uOmega^{p-q,+}_{X/S}\otimes f^*\Omega^q_S, \quad \quad \bbG^{q,-}\uOmega^p_X\qis \uOmega^{p-q,-}_{X/S}\otimes f^*\Omega^q_S.
\]

\item $\bbF^{q,+}\uOmega^p_X$ and $\bbF^{q,-}\uOmega^p_X$ are functorial.

\item If $f$ is smooth, then $\uOmega^{p,+}_{X/S}\qis \uOmega^{p,-}_{X/S}\qis \Omega^p_{X/S}$.

\item Let $\{\nu_i: U_i\into X\}$ be a finite open cover of $X$. Then
\[
\uOmega^{p,\pm}_{X/S} \qis R(\nu_\kdot)_* \uOmega^{p,\pm}_{U_\kdot /S}.
\]
\end{enumerate}
\end{theorem}

Note that the statement for the left relative Du Bois complex $\uOmega^{p,-}_{X/S}$ follows from \cite[Theorem 3.4]{Kov97}. Our construction for the right relative complex agrees with the one constructed in \cite{Nin26} over a smooth curve, see Proposition \ref{prop_curve_compare}.

We also show base change properties for this generalized right relative Du Bois complex for fibers that are strict higher Du Bois (cf. Definition \ref{def_higher_DB}).

\begin{theorem}( = Theorem \ref{base change})
Let $f: X\to S$ be a flat morphism of algebraic varieties with $\dim X=n$, $\dim S=k$ and $S$ smooth. Let $s\in S$ be a closed point and suppose the fiber $X_s$ has strict-$m$-Du Bois singularities. Then $X_s$ satisfies right-$m$-base change (cf. Definition \ref{def_bc}) for the relative Du Bois complex, that is, for all $0\le p\le m$,
\[
\uOmega^{p, +}_{X/S}\otimes^{\tL}_X\scrO_{X_s} \qis \uOmega^p_{X_s}.
\]
\end{theorem}

We remark that a recent work of Dirks in \cite{Dir26} constructs an alternative relative Du Bois complex using the theory of mixed Hodge modules, over smooth bases of arbitrary dimensions. It behaves similarly to the right relative Du Bois complex, and Dirks asks \cite[\S 5, Questions 3, 4]{Dir26} whether the existing constructions agree. This result will be discussed in an upcoming work \cite{Nin26a}.

\smallskip
\subsection*{Paper Outline}
Section \ref{prelim} reviews preliminary material essential to the construction. In Section \ref{section rel db}, we construct the right relative Du Bois complex $\uOmega^{p,+}_{X/S}$ directly for a singular total space, refining the methods of \cite{Kov97}. We verify that both relative Du Bois complexes $\uOmega^{p, \pm}_{X/S}$ agree with the existing definitions in the literature \cite{Kov96, Nin26} over a smooth curve. Section \ref{section base change} shows the base change property of $\uOmega^{p,+}_{X/S}$ for flat families with strict higher Du Bois fibers.

\smallskip
\subsection*{Acknowledgments} 

The authors would like to thank S\'andor Kov\'acs and Brad Dirks for helpful discussions.

\smallskip
\subsection*{AI Disclosure} 

No generative AI is used for the mathematical contents of this paper.

\medskip
\section{Preliminaries}\label{prelim}

\subsection*{The Du Bois Complex} The Du Bois complex was originally constructed in \cite{DB81} using simplicial hyperresolutions. In this paper, by hyperresolution we will mean cubical hyperresolutions in the sense of \cite{GNPP88}. We refer the reader there for detail, see also \cite[\S 10.6]{Kol13}.

\begin{theorem}[\cite{DB81,GNPP88}]\label{thm_DB}
Let $X$ be a reduced scheme of finite type over $\bbC$. Then there exists an object $\uOmega^\kdot_X$ in $D^b_{\rm filt, coh}(X)$, with associated graded $\uOmega^p_X = \Gr^p_{\rm filt}\uOmega^\kdot_X[p]$ satisfying the following properties.
\begin{enumerate}
\item If $\pi_\kdot : X_\kdot\to X$ is any hyperresolution of $X$, then 
\[
\uOmega^\kdot_X \qis R\pi_{\kdot *} \Omega^\kdot_{X_\kdot}.
\]

\item The construction is functorial, i.e. if $\phi: Y\to X$ is a morphism of reduced schemes of finite type over $\bbC$, then there exists a natural map $\phi^*$ of filtered complexes
$$
\phi^*: \uOmega^\kdot_X \to R\phi_*\uOmega^\kdot_Y.
$$
Furthermore, if $\phi$ is proper, then $\phi^*$ is a morphism in $D^b_{\rm filt,coh}(X)$.

\item \label{dumb filtration} Let $\Omega^\kdot_X$ be the usual De Rham complex of K\"ahler differentials considered with the `filtration b\^ete'. Then there exists a natural map of filtered complexes
$$
\Omega^{\kdot}_X \to \uOmega^\kdot_X
$$
and if $X$ is smooth, this map is a quasi-isomorphism.

\item Let $\{\nu_i : U_i \hookrightarrow X\}$ be a finite open cover of $X$, with $\nu_\kdot: U_\kdot \to X$ the associated map. Then
$$
\uOmega^\kdot_X \qis R\nu_{\kdot *}\uOmega^\kdot_{U_\kdot}.
$$
\end{enumerate}
\end{theorem}

The definitions of higher Du Bois singularities have several distinct flavors outside of the local complete intersection setting. As observed in \cite{Nin26}, the notion of strict-$m$-Du Bois singularities appears to be the most suitable for base change type questions, so we only include this variant here. We refer the reader to \cite{SVV23,Kov25,NN25} for detailed discussions on other notions.

\begin{definition}\label{def_higher_DB}
Let $X$ be a reduced scheme of finite type over $\bbC$ and let $m\in \bbN$. We say $X$ has strict-$m$-Du Bois singularities if the natural map
\[
\Omega^p_X \xrightarrow{\qis} \uOmega^p_X
\]
is a quasi-isomorphism for each $0\le p\le m$.
\end{definition}

\subsection*{Hyperfiltrations}

Let $\mathfrak{A}$ be an abelian category and $D(\mathfrak{A})$ its derived category. Let $\Gamma : \mathfrak{A} \to \mathfrak{A} \mathfrak{b}$ be a left exact additive functor from $\mathfrak{A}$ to the category of abelian groups.

\begin{definition}\label{hyperfiltration} \cite[1.2.1]{Kov97}
Let $K \in \operatorname{Ob}(D^b(\mathfrak{A}))$ be a bounded complex. A hyperfiltration $\bbF$ of $K$ consists of a set of objects $\bbF^q K \in \operatorname{Ob}(D^b(\mathfrak{A}))$ for $q = l, \cdots , k + 1$, where $l, k \in \bbZ$ and morphisms
$$
\phi_q \in \Hom_{D^b(\mathfrak{A})}(\bbF^{q+1} K, \bbF^q K) \quad \text{ for } q = l, \cdots, k,
$$
with $\bbF^l\qis K$ and $\bbF^{k+1}\qis 0$. The $p$-th associated graded complex of a hyperfiltration $\bbF$ is
$$
\bbG^p := \Gr^p_{\bbF}K = \operatorname{Cone(\phi_p)}.
$$
This is the class of the mapping cone of a morphism in $C(\mathfrak{A})$ whose class in $D(\mathfrak{A})$ is $\phi_p$, so that $\bbF^{p+1} \xrightarrow[]{}\bbF^p \xrightarrow{} \bbG^p \xrightarrow[]{+1}$ forms a distinguished triangle. 
\end{definition}

\section{The right relative Du Bois complex}\label{section rel db}

We extend the construction of the right relative Du Bois complex $\uOmega^{p,+}_{X/B}$ from \cite{Nin26}, where the base is a smooth curve, to a smooth base $S$ of arbitrary dimension. This is the ascending analogue of the construction of \cite{Kov97}, in which the pieces $\uOmega^{p}_{X/S}$ are built upward from $p=0$ rather than downward from $p=\dim X - \dim S$.

Our approach further simplifies the constructions in \emph{loc. cit.} by building $\uOmega^p_{X/S}$ directly without assuming $X$ is smooth. The central observation is that the only data required from a hyperresolution is a wedge map on the Du Bois complex of $X$. The following lemma provides the naturality that we need.

\begin{lemma}\label{lem_wedge_uOmega}
Let $f:X\to S$ be a morphism of algebraic varieties, with $S$ smooth. Then, for each $p, q\ge 0$, there exists a functorial wedge map
\[
\wedge_{p,q}: \uOmega^{p}_X\otimes f^*\Omega^{q}_S \to \uOmega^{p+q}_X.
\]
Furthermore, the following diagram commutes
\[\begin{tikzcd}
\uOmega^p_X\otimes f^*\Omega^q_S \otimes f^*\Omega^r_S \ar[r] \ar[d, "\wedge_{p,q}"] &
\uOmega^p_X\otimes f^*\Omega^{q+r}_S \ar[d, "\wedge_{p, q+r}"] \\
\uOmega^{p+q}_X \otimes f^*\Omega^r_S \ar[r, "\wedge_{p+q, r}"] &
\uOmega^{p+q+r}_X.
\end{tikzcd}\]
\end{lemma}

\begin{proof}
This is a simple extension of the observations in \cite[\S 1]{Kov96}, see also \cite[Construction 3.1]{Nin26}. Fix a cubical hyperresolution $\pi_\kdot: X_\kdot \to X$, so that we have a specific representative of $\uOmega^p_X$, namely
\[
R(\pi_{\kdot})_*\Omega^p_{X_\kdot}\qis \uOmega^p_X.
\]
Using the projection formula to identify forms pulled back from $S$, we may define
\[
\wedge_{p,q}: R(\pi_{\kdot})_*\Omega^p_{X_\kdot}\otimes f^*\Omega^q_S \to R(\pi_{\kdot})_*\Omega^p_{X_\kdot}
\]
determined by
\[
\wedge_{p,q} ((\pi_{\alpha})_*\eta_\alpha \otimes \xi) = (\pi_\alpha)_* (\eta_\alpha \wedge (\pi_\alpha)^* \xi),
\]
where $\pi_\alpha:X_\alpha \to X$ is a component of the hyperresolution chosen, and $\eta_\alpha, \xi$ are the appropriate differential forms. It is then clear that the diagram in the statement of the lemma commutes with the chosen representative. 

We now only need to show that $\wedge_{p,q}$ is independent of the hyperresolution chosen. Suppose $\widetilde{\pi}_{\kdot}:\widetilde{X}_\kdot \to X$ is another cubical hyperresolution and $\alpha:X_\kdot \to \widetilde{X}_\kdot$ is a morphism between them. Then we have a commutative diagram
\[\begin{tikzcd}
R(\pi_\kdot)_* \Omega^p_{X_\kdot} \otimes f^*\Omega^q_S \ar[r, "\wedge_{p,q}"] \ar[d, "\qis"] &
R(\pi_\kdot)_* \Omega^{p+q}_{X_\kdot} \ar[d, "\qis"] \\
R(\widetilde{\pi}_\kdot)_* \Omega^p_{\widetilde{X}_\kdot} \otimes f^*\Omega^q_S \ar[r, "\widetilde{\wedge}_{p,q}"] &
R(\widetilde{\pi}_\kdot)_* \Omega^{p+q}_{\widetilde{X}_\kdot}
\end{tikzcd}\]
where the vertical arrows are quasi-isomorphisms induced by $\alpha$, see \cite[Lemme 2.1.4]{DB81}. Therefore, the equivalence class of $\wedge_{p,q}$ in $D(X)$ is independent of the chosen hyperresolution, as desired.
\end{proof}

The following construction is an analogue to the induction presented in \cite{Kov97}, built upward to match the construction presented in \cite{Nin26}. 

\begin{construction}\label{cons_frakF}
Let $f:X \to S$ be a morphism of algebraic varieties, with $S$ smooth of dimension $k$. With induction on $i\ge 0$, for each fixed $i$, our goal is to define auxiliary objects $\frakF^{p,+}_i$ with maps
\[
\mu^{p,q}_i: \frakF^{p,+}_i \otimes f^* \Omega^q_S\to \frakF^{p+q,+}_i,
\]
for all $p, q\ge 0$. We further require at each step that $\mu^{p,q}_i$ is compatible with wedge maps from the base, meaning that the diagram
\[\begin{tikzcd}
{\frakF^{r,+}_{i} \otimes f^*\Omega_S^{p-r} \otimes f^*\Omega_S^q} & {\frakF^{r,+}_{i} \otimes f^*\Omega_S^{p+q-r}} \\
{\frakF^{p,+}_{i}\otimes f^*\Omega_S^q} & {\frakF^{p+q,+}_{i}}
\ar[from=1-1, to=1-2]
\ar["\mu_i^{r, p-r}", from=1-1, to=2-1]
\ar["\mu_i^{r, p+q-r}", from=1-2, to=2-2]
\ar["{\mu_{i}^{p,q}}"', from=2-1, to=2-2]
\end{tikzcd}\]
commutes for all $p, q, r\ge 0$.

We begin with the $i=0$ case. Let $p\ge 0$ and define the first term to be $ \frakF^{p,+}_0 := \uOmega^p_X$, with 
\[
\mu_0^{p,q}:= \wedge_{p,q}:\uOmega^p_X\otimes f^*\Omega^q_S\to \uOmega^{p+q}_X.
\]
Define $\frakF^{p,+}_1:= \Cone (\mu_0^{0,p})$, that is, it sits in the following distinguished triangle
\begin{equation}\label{first triangle}
\frakF^{0,+}_{0} \otimes f^*\Omega_S^{p} \xrightarrow{\mu^{0,p}_0} \frakF^{p,+}_{0} \longrightarrow \frakF^{p,+}_1 \xrightarrow{+1}.
\end{equation}
For any $q\ge 0$, we obtain a commutative diagram:

\begin{equation*}
\begin{tikzcd}
{\uOmega_X^0 \otimes f^*\Omega_S^{p} \otimes f^*\Omega_S^q} & {\uOmega_X^0\otimes f^*\Omega_S^{p+q}} \\
{\uOmega_X^p \otimes f^*\Omega_S^{q}} & {\uOmega_X^{p+q}} \\
{\frakF^{p,+}_1 \otimes f^*\Omega_S^q} \ar[d, "+1"]& {\frakF^{p+q,+}_1} \ar[d, "+1"]\\
\, & \,
\ar[from=1-1, to=1-2]
\ar[from=1-1, to=2-1]
\ar[from=1-2, to=2-2]
\ar["{\mu_0^{p,q}}"', from=2-1, to=2-2]
\ar[from=2-1, to=3-1]
\ar[from=2-2, to=3-2]
\ar["{\mu_1^{p,q}}"', dashed, from=3-1, to=3-2]
\end{tikzcd}
\end{equation*}
The two columns are distinguished triangles of the form \ref{first triangle} above, and the compatibility is given by Lemma \ref{lem_wedge_uOmega}. We see that it induces a map $\mu_1^{p,q} : \frakF^{p,+}_1 \otimes f^*\Omega^q_S \to \frakF^{p+q,+}_1$ on the bottom. This gives the first step of the induction.

Suppose the inductive hypothesis holds for some $i\ge 0$. In particular, there are maps
\[
\mu_i^{p,q}:\frakF^{p,+}_i \otimes f^*\Omega^q_S \to \frakF^{p+q,+}_i
\]
for all $p,q \geq 0$ that commutes with the wedge maps on the base. Define $\frakF^{p,+}_{i+1} := \Cone(\mu_i^{i, p-i})$, so that
\begin{equation}\label{triangle}
\frakF^{i,+}_{i} \otimes f^*\Omega_S^{p-i} \xrightarrow{\mu^{i,p-i}_i} \frakF^{p,+}_{i} \longrightarrow \frakF^{p,+}_{i+1} \xrightarrow{+1}
\end{equation}
is a distinguished triangle. Note that for $p< i$ this gives $\frakF^{p,+}_{i} \qis \frakF^{p,+}_{i+1}$ since by convention $\Omega_S^{p-i}=0$ for $p<i$. Similar to before, this gives the diagram with exact columns of the form (\ref{triangle})

\begin{equation}\label{box}
\begin{tikzcd}
{\frakF^{i,+}_{i} \otimes f^*\Omega_S^{p-i} \otimes f^*\Omega_S^q} & {\frakF^{i,+}_{i} \otimes f^*\Omega_S^{p+q-i}} \\
{\frakF^{p,+}_{i}\otimes f^*\Omega_S^q} & {\frakF^{p+q,+}_{i}} \\
{\frakF^{p,+}_{i+1} \otimes f^*\Omega_S^q} \ar[d, "+1"] & {\frakF^{p+q,+}_{i+1}}. \ar[d, "+1"]\\
\, & \,
\ar[from=1-1, to=1-2]
\ar[from=1-1, to=2-1]
\ar[from=1-2, to=2-2]
\ar["{\mu_{i}^{p,q}}"', from=2-1, to=2-2]
\ar[from=2-1, to=3-1]
\ar[from=2-2, to=3-2]
\ar["{\mu_{i+1}^{p,q}}"', dashed, from=3-1, to=3-2]
\end{tikzcd}
\end{equation}
The top square commutes by the inductive hypothesis, therefore we obtain the induced bottom map to be $\mu_{i+1}^{p,q}: \frakF^{p,+}_{i+1} \otimes f^*\Omega^q_S \to \frakF^{p+q,+}_{i+1}$. To check the required compatibility, consider the following diagram:

\begin{equation*}
\begin{tikzcd}[cramped, column sep = tiny]
& {\frakF^{i,+}_{i} \otimes f^*\Omega_S^{p-i-1} \otimes f^*\Omega_S^1 \otimes f^*\Omega_S^q} && {\frakF^{i,+}_{i}  \otimes f^*\Omega_S^1 \otimes f^*\Omega_S^{p+q-i-1}} \\
& {\frakF^{i+1,+}_{i} \otimes f^*\Omega_S^{p-i-1} \otimes f^*\Omega_S^q} && {\frakF^{i+1,+}_{i} \otimes f^*\Omega_S^{p+q-i-1}} \\
& {\frakF^{i+1,+}_{i+1} \otimes f^*\Omega_S^{p-i-1} \otimes f^*\Omega_S^q} && {\frakF^{i+1,+}_{i+1} \otimes f^*\Omega_S^{p+q-i-1}} \\
{\frakF^{i,+}_{i} \otimes f^*\Omega_S^{p-i} \otimes f^*\Omega_S^q} & {} & {\frakF^{i,+}_{i} \otimes f^*\Omega_S^{p+q-i}} & {} \\
{\frakF^{p,+}_{i}\otimes f^*\Omega_S^q} && {\frakF^{p+q,+}_{i}} \\
{\frakF^{p,+}_{i+1} \otimes f^*\Omega_S^q} && {\frakF^{p+q,+}_{i+1}} \\
{} && {}
\arrow[from=1-2, to=1-4]
\arrow[from=1-2, to=2-2]
\arrow[from=1-2, to=4-1]
\arrow[from=1-4, to=2-4]
\arrow[from=1-4, to=4-3]
\arrow[from=2-2, to=2-4]
\arrow[from=2-2, to=3-2]
\arrow["{\mu_{i}^{i+1, p-i-1}}"{description}, from=2-2, to=5-1]
\arrow[from=2-4, to=3-4]
\arrow["{\mu_{i}^{i+1, p+q-i-1}}"{description}, from=2-4, to=5-3]
\arrow[dashed, from=3-2, to=3-4]
\arrow["{+1}", from=3-2, to=4-2]
\arrow["{\mu_{i+1}^{i+1, p-i-1}}"{description}, dashed, from=3-2, to=6-1]
\arrow["{+1}", from=3-4, to=4-4]
\arrow["{\mu_{i+1}^{i+1, p+q-i-1}}"{description, pos=0.6}, dashed, from=3-4, to=6-3]
\arrow[from=4-1, to=4-3]
\arrow[from=4-1, to=5-1]
\arrow[from=4-3, to=5-3]
\arrow["{\mu_{i}^{p,q}}"{description}, from=5-1, to=5-3]
\arrow[from=5-1, to=6-1]
\arrow[from=5-3, to=6-3]
\arrow["{\mu_{i+1}^{p, q}}"{description}, dashed, from=6-1, to=6-3]
\arrow["{+1}", from=6-1, to=7-1]
\arrow["{+1}", from=6-3, to=7-3]
\end{tikzcd}
\end{equation*}

The vertical faces commute as they are all diagrams of the form (\ref{box}). The top square commutes by associativity of the wedge product, and the middle square commutes by the inductive hypothesis. Therefore, the bottom square is commutative on any chosen representative. For clarity, the bottom square reads:

\[\begin{tikzcd}
{\frakF^{i+1,+}_{i+1} \otimes f^*\Omega_S^{p-i-1} \otimes f^*\Omega_S^q} & {\frakF^{i+1,+}_{i+1} \otimes f^*\Omega_S^{p+q-i-1}} \\
{\frakF^{p,+}_{i+1} \otimes f^*\Omega_S^q} & {\frakF^{p+q,+}_{i+1}}
\ar[from=1-1, to=1-2]
\ar["{\mu^{i+1,p-i-1}_{i+1}}"', from=1-1, to=2-1]
\ar["{\mu^{i+1,p+q-i-1}_{i+1}}", from=1-2, to=2-2]
\ar["{\mu^{p,q}_{i+1}}"', from=2-1, to=2-2]
\end{tikzcd}\]
This is precisely the compatibility required for the inductive step. Therefore, we conclude the construction of $\frakF^{p,+}_{i}$ and the necessary accompanying maps $\mu_i^{p,q}$ for all $i, p, q\ge 0$. We summarize this information as:
\end{construction}

\begin{definition}\label{defs_frakF}
Following Construction \ref{cons_frakF}, for $p, i\ge 0$, we define
\[
\frakF^{p,+}_{i+1} := \operatorname{Cone}(\frakF^{i,+}_{i} \otimes f^*\Omega_S^{p-i} \xrightarrow{\mu^{i,p-i}_i} \frakF^{p,+}_{i} ).
\]
By convention, we put $\frakF^{p,+}_{i} =0$ if $p< 0$.
Note that $\frakF^{p,+}_0 = \uOmega^p_X$. Furthermore,
\[
0\qis \frakF^{p,+}_{p+1} \qis \frakF^{p,+}_i
\]
for all $i> p$.
\end{definition}

This allows us to define the right relative Du Bois complex.

\begin{definition}\label{def_rel_db} 
Let $f:X\to S$ be a morphism of varieties with $S$ smooth. We define $\uOmega^{p,+}_{X/S} \in \operatorname{Ob}(D(X))$, called the \emph{right relative Du Bois complex} of $f:X\to S$, as follows. Put
\[
\uOmega^{p,+}_{X/S} = \begin{cases}
0 & p<0, \\
\frakF^{p,+}_p & p\ge 0.
\end{cases}
\]
In particular, the triangle in (\ref{triangle}) reads
\[
\uOmega^{i,+}_{X/S} \otimes f^*\Omega^{p-i}_S \to \frakF^{p,+}_i \to \frakF^{p,+}_{i+1} \xrightarrow{+1}.
\]
\end{definition}

We use the objects defined in Definition \ref{defs_frakF} to construct a hyperfiltration of $\uOmega^p_X$.

\begin{construction}\label{cons_hyperfilt}
As in Construction \ref{cons_frakF}, let $f:X\to S$ be a morphism of varieties with $S$ smooth of dimension $k$. Note that by the iterative definition of $\frakF^{p,+}_i$, there is a natural map $\frakF^{p,+}_i \to \frakF^{p,+}_{i+1}$ given by the second map of the triangle (\ref{triangle}). These yield a chain
\[
\uOmega^p_X\qis \frakF^{p,+}_0 \to \frakF^{p,+}_1 \to \dots \to \frakF^{p,+}_p \qis \uOmega^{p,+}_{X/S} \to \frakF^{p,+}_{p+1} \qis 0.
\]
For $0\le q\le p+1$, define
\[
\bbF^{q,+}_S\uOmega^p_X := \operatorname{Cone}\big(\uOmega^p_X \to \frakF^{p,+}_{p-q+1}\big)[-1].
\]
If no confusion will arise, we will suppress the notation of the base $S$ and denote it by $\bbF^{q,+}$. Note that 
\[
\bbF^{0,+}\uOmega^p_X\qis \uOmega^p_X \quad \text{and} \quad \bbF^{p+1,+}\uOmega^p_X =0.
\]
By convention, we also define $\bbF^{q,+}\uOmega^p_X=0$ for $q> p+1$. Now, consider the following diagram:

\[\begin{tikzcd}
{\mathbb{F}^{q+1, +}\uOmega^p_X} & {\uOmega^p_X} & {\mathfrak{F}^{p,+}_{p-q}} \ar[r,"+1"] & \, \\
{\mathbb{F}^{q,+}\uOmega^p_X} & {\uOmega^p_X} & {\mathfrak{F}^{p,+}_{p-q+1}} \ar[r,"+1"] & \, \\
{\mathbb{G}^{q,+}\uOmega^p_X} \ar[d,"+1"] & 0 \ar[d,"+1"] & {(\mathfrak{F}^{p-q,+}_{p-q}\otimes f^*\uOmega^q_S)[1]} \ar[r,"+1"] \ar[d,"+1"] & \, \\
\, & \, & \, & \,
\ar[from=1-1, to=1-2]
\ar[dotted, from=1-1, to=2-1]
\ar[from=1-2, to=1-3]
\ar[equal, from=1-2, to=2-2]
\ar[from=1-3, to=2-3]
\ar[from=2-1, to=2-2]
\ar[from=2-1, to=3-1]
\ar[from=2-2, to=2-3]
\ar[from=2-2, to=3-2]
\ar[from=2-3, to=3-3]
\ar[from=3-1, to=3-2]
\ar[from=3-2, to=3-3]
\end{tikzcd}\]

The third column is the shifted triangle (\ref{triangle}) at $i = p -q$, and the top two rows are triangles by construction. By commutativity of the top right square, we may extend to obtain $\bbF^{q+1,+}\uOmega^p_X\to\bbF^{q,+}\uOmega^p_X$, and thereafter define the associated graded complex $\bbG^{q,+}\uOmega^p_X$ to be its cone. The derived nine-lemma gives exactness of the third row, which implies that
\begin{equation}\label{eqn_graded}
\bbG^{q,+}\uOmega^p_X\qis\uOmega^{p-q,+}_{X/S}\otimes f^*\Omega_S^q.
\end{equation}
\end{construction}

We will take a quick detour to establish the functorial properties of $\bbF^{\kdot,+}_S\uOmega^p_X$.

\begin{lemma} \label{functoriality}
Let $\phi: Y\to X$ be an $S$-morphism, and $g=f\circ\phi$. Then for each $p, q\ge 0$, there are natural maps of hyperfiltrations
\[
\bbF^{q,+}_S\uOmega^p_X \to R\phi_* \bbF^{q,+}_S\uOmega^p_Y.
\]
That is, the maps are compatible with the hyperfiltration and descend to the associated graded. In particular, there are natural maps
\[
\uOmega^{p,+}_{X/S} \to R\phi_* \uOmega^{p,+}_{Y/S}.
\]
\end{lemma}
\begin{proof}
We will temporarily use the notation $\frakF^{p,+}_{i, X}$ and $\frakF^{p,+}_{i,Y}$ for the auxiliary object in Definition \ref{defs_frakF} associated to $X$ and $Y$, respectively. Using the triangle (\ref{triangle}), consider the following diagram
\[\begin{tikzcd}
\frakF^{i,+}_{i, X}\otimes f^*\Omega_S^{p-i} \ar[r, "\mu_{i, X}^{i, p-i}"] \ar[d] &
\frakF^{p,+}_{i,X} \ar[r] \ar[d] &
\frakF^{p,+}_{i+1,X}\ar[d, dotted] \ar[r, "+1"] &
\, \\
R\phi_*\frakF^{i,+}_{i,Y}\otimes f^*\Omega_S^{p-i} \ar[r, "R\phi_*\mu_{i,Y}^{i, p-i}"] &
R\phi_*\frakF^{p,+}_{i,Y} \ar[r] &
R\phi_*\frakF^{p,+}_{i+1,Y}. \ar[r, "+1"] &
\,
\end{tikzcd}\]
Note the use of the projection formula to identify the bottom left object. For $i=0$, the vertical arrows are given by the map of Du Bois complexes, as $\frakF^{p,+}_{0,X}\qis \uOmega^p_X$ and similarly for $Y$, see Theorem \ref{thm_DB}(2). Furthermore, the left square commutes by the functoriality of the wedge map $\mu^{p,q}_i=\wedge_{p,q}$ in Lemma \ref{lem_wedge_uOmega}. Therefore, the same inductive argument on $i\ge 0$ as in Construction \ref{cons_frakF} shows that we may extend iteratively to obtain a natural map
\[
\frakF^{p,+}_{i+1,X} \to R\phi_*\frakF^{p,+}_{i+1,Y}.
\]
Moreover, it is compatible with the map on Du Bois complexes $\uOmega^p_{X}\to R\phi_* \uOmega^p_Y$. This allows us to extend the map to the hyperfiltrations, as in the following diagram:
\[\begin{tikzcd}
{\mathbb{F}^{q, +}\uOmega^p_X} &
{\uOmega^p_X} &
{\mathfrak{F}^{p,+}_{p-q+1, X}} \ar[r,"+1"] &
\, \\
{R\phi_*\mathbb{F}^{q,+}\uOmega^p_Y} &
{R\phi_*\uOmega^p_Y} &
{R\phi_*\mathfrak{F}^{p,+}_{p-q+1,Y}} \ar[r,"+1"] &
\,
\ar[from=1-1, to=1-2]
\ar[dotted, from=1-1, to=2-1]
\ar[from=1-2, to=1-3]
\ar[from=1-2, to=2-2]
\ar[from=1-3, to=2-3]
\ar[from=2-1, to=2-2]
\ar[from=2-2, to=2-3]
\end{tikzcd}\]
This gives the first statement of the lemma. Applying the triangle for the associated graded in Construction \ref{cons_hyperfilt} gives the compatibility. Lastly, the map $\uOmega^{p,+}_{X/S}\to R\phi_* \uOmega^{p,+}_{Y/S}$ is simply given by $\frakF^{p,+}_{p,X}\to R\phi_*\frakF^{p,+}_{p,Y}$. This completes the proof.
\end{proof}

\begin{lemma} \label{indephp}
Let $\pi_\kdot: X_\kdot \to X$ be a cubical hyperresolution of $X$. Then for each $p\ge 0$,
\[
R(\pi_\kdot)_*\uOmega^{p, +}_{X_\kdot/S} \qis \uOmega^{p, +}_{X/S}.
\]
\end{lemma}

\begin{proof}
By (\ref{eqn_graded}) and Lemma \ref{functoriality} we obtain a morphism of distinguished triangles:

\[\begin{tikzcd}
\bbF^{q+1, +}\uOmega^p_{X} \ar[r] \ar[d] &
\bbF^{q, +}\uOmega^p_{X} \ar[r] \ar[d] &
\uOmega^{p-q,+}_{X/S}\otimes f^*\Omega_S^q \ar[r, "+1"] \ar[d] &
\, \\
(R\pi_\kdot)_*\bbF^{q+1, +}\uOmega^p_{X_\kdot} \ar[r] &
(R\pi_\kdot)_*\bbF^{q, +}\uOmega^p_{X_\kdot} \ar[r] &
(R\pi_\kdot)_*(\uOmega^{p-q,+}_{X_\kdot/S})\otimes f^*\Omega_S^q. \ar[r, "+1"] &
\,
\end{tikzcd}\]
Note the use of the projection formula for the last term. The proof follows from the same inductive argument as in Lemma \ref{functoriality} and the five-lemma, reducing down to the statement about the Du Bois complex, $\uOmega^p_X\qis R(\pi_\kdot)_*\Omega^p_{X_\kdot}$ (cf. \cite[2.3]{DB81} \cite[V.3.3]{GNPP88}). See also \cite[Lemma 3.3]{Kov97} for an analogous statement.
\end{proof}

We collect the results in this section in the following theorem. For easy comparison, we also include the left relative Du Bois complex by \cite{Kov97}.

\begin{theorem}\label{thm_relDB}
Let $f: X\to S$ be a morphism of algebraic varieties with $\dim X=n$, $\dim S=k$ and $S$ smooth. Then there exist objects $\uOmega^{p, +}_{X/S}, \uOmega^{p, -}_{X/S}$ in $D^b_{\rm coh}(X)$ for all $p\in \bbZ$ with the following properties:
\begin{enumerate}
\item \label{reldb1} $\uOmega^{p,+}_{X/S}\qis 0$ for $p< 0$, $\uOmega^{p, -}_{X/S}\qis 0$ for $p>n-k$.

\item \label{reldb2} There exist hyperfiltrations $\bbF^{q,+}\uOmega^p_X, \bbF^{q,-}\uOmega^p_X$ of $\uOmega^p_X$ with $0\le q\le k+1$, such that
\[
\bbF^{0,+}\uOmega^p_X \qis \bbF^{0,-}\uOmega^p_X \qis \uOmega^p_X, \quad \bbF^{k+1,+}\uOmega^p_X \qis \bbF^{k+1,-}\uOmega^p_X \qis 0,
\]
whose associated graded complexes are, with $0\le q\le k$,
\[
\bbG^{q,+}\uOmega^p_X\qis \uOmega^{p-q,+}_{X/S}\otimes f^*\Omega^q_S, \quad \bbG^{q,-}\uOmega^p_X\qis \uOmega^{p-q,-}_{X/S}\otimes f^*\Omega^q_S.
\]

\item \label{reldb3} $\bbF^{q,+}\uOmega^p_X$ and $\bbF^{q,-}\uOmega^p_X$ are functorial, i.e. if $\phi: Y \to X$ is an S-morphism, then there are natural maps $\bbF^q\uOmega^p_X \to R\phi_*\bbF^q\uOmega^p_Y$ in $D(X)$ forming the commutative diagram of distinguished triangles:
\[\begin{tikzcd}
{\mathbb{F}^{q+1, \pm}\uOmega^p_X} & {\mathbb{F}^{q, \pm}\uOmega^p_X} & {\mathbb{G}^{q, \pm}\uOmega^p_X} & {} \\
{R\phi_*\mathbb{F}^{q+1, \pm}\uOmega^p_Y} & {R\phi_*\mathbb{F}^{q, \pm}\uOmega^p_Y} & {R\phi_*\mathbb{G}^{q, \pm}\uOmega^p_Y} & {}
\ar[from=1-1, to=1-2]
\ar[from=1-1, to=2-1]
\ar[from=1-2, to=1-3]
\ar[from=1-2, to=2-2]
\ar["{+1}", from=1-3, to=1-4]
\ar[from=1-3, to=2-3]
\ar[from=2-1, to=2-2]
\ar[from=2-2, to=2-3]
\ar["{+1}", from=2-3, to=2-4]
\end{tikzcd}\]

In particular, there are natural maps in D(X) forming the commutative diagram:

\[\begin{tikzcd}
{\uOmega^{p}_{X}} & {\uOmega^{p, \pm}_{X/S}} \\
{R\phi_*\uOmega^{p}_{Y}} & {R\phi_*\uOmega^{p, \pm}_{Y/S}}
\ar[from=1-1, to=1-2]
\ar[from=1-1, to=2-1]
\ar[from=1-2, to=2-2]
\ar[from=2-1, to=2-2]
\end{tikzcd}\]

\item \label{reldb4} If $f$ is smooth, then $\uOmega^{p,+}_{X/S}\qis \uOmega^{p,-}_{X/S}\qis \Omega^p_{X/S}$.

\item\label{reldb5} Let $\{\nu_i: U_i\into X\}$ be a finite open cover of $X$. Then
\[
\uOmega^{p,\pm}_{X/S} \qis R(\nu_\kdot)_* \uOmega^{p,\pm}_{U_\kdot /S}.
\]
\end{enumerate}
\end{theorem}

\begin{proof}
The statements for $\uOmega^{p,-}_{X/S}$ follow from \cite[Theorem 3.4]{Kov97}. For the right complex $\uOmega^{p,+}_{X/S}$, its existence and (\ref{reldb1}) follow from Construction \ref{cons_frakF} and Definition \ref{def_rel_db}. (\ref{reldb2}) is Construction \ref{cons_hyperfilt} and (\ref{eqn_graded}). (\ref{reldb3}) is Lemma \ref{functoriality}. The rest are standard consequences of the analogous properties of the absolute Du Bois complex and the constructions of the hyperfiltrations, see \cite[Theorem 1.3]{Kov96}, \cite[Theorem 3.4]{Kov97}, \cite[Theorem 3.2]{Nin26}.
\end{proof}

\begin{corollary}
Let $f:X\to S$ be a morphism of algebraic varieties such that $S$ is smooth. Then for any $p\ge 0$ and any locally free sheaf $\scrE$ on $X$ there exists a spectral sequence
\[
E^{q,i}_1 = \bbH^{q+i}(X, \uOmega^{p-q, \pm}_{X/S} \otimes f^*\Omega^q_S\otimes \scrE) \Rightarrow \bbH^{q+i}(X, \uOmega^p_X\otimes \scrE).
\]
\end{corollary}

\begin{proof}
The statement for $\uOmega^{p,-}_{X/S}$ is \cite[Corollary 3.5]{Kov97}. The analogous statement for $\uOmega^{p,+}_{X/S}$ follows immediately from Theorem \ref{thm_relDB}(\ref{reldb2}) and \cite[Theorem 1.2.2]{Kov97}.
\end{proof}

The next proposition shows that the relative Du Bois complexes $\uOmega^{p,\pm}_{X/S}$ agree with the existing definition in the literature when the base is a smooth curve (cf. \cite{Kov96, Nin26}), justifying the notation. We include the comparison for the left relative Du Bois complex of Kov\'acs here for completeness, as it is not included in \cite{Kov97}.

\begin{proposition}\label{prop_curve_compare}
Let $f:X \to S$ be a morphism of algebraic varieties, where $S$ is a smooth curve. Then $\uOmega^{p, +}_{X/S}, \uOmega^{p,-}_{X/S}$ agree with the objects defined in \cite[Theorem 3.2]{Nin26} for all $p\in \bbN$.
\end{proposition}

\begin{proof}
To avoid confusion, we will always use the notations $\uOmega^{p, +}_{X/S}, \uOmega^{p,-}_{X/S}$ to mean the objects constructed in Theorem \ref{thm_relDB}, and use $\uuOmega^{p, +}_{X/S}, \uuOmega^{p,-}_{X/S}$ to temporarily denote those constructed in \cite[Theorem 3.2]{Nin26}. The statement trivially holds if $p<0$ for $\uOmega^{p,+}_{X/S}$ and if $p> \dim X- \dim f(X)$ for $\uOmega^{p,-}_{X/S}$. So let $p$ be in the appropriate range where $\frakF^{p,+}_{i}, \frakF^{p,-}_{i}$ are iteratively defined as cones.

Since $\dim S=1$, $\Omega^q_S=0$ for $q>1$. We see that for all $0\le i\le p-2$,
\[
\frakF^{p,+}_{i+1} = \Cone (\frakF^{i,+}_i\otimes f^* \Omega^{p-i}_S \to \frakF^{p,+}_i) \qis \frakF^{p,+}_i.
\]
In other words, for all $0\le i\le p-1$
\[
\frakF^{p,+}_{i}\qis \frakF^{p,+}_{0} \qis \uOmega^p_X.
\]
Therefore the triangle for $\uOmega^{p,+}_{X/S} =\frakF^{p,+}_p$ reads
\[
\uOmega^{p-1,+}_{X/S}\otimes f^*\Omega^1_S \to \uOmega^p_X \to \uOmega^{p,+}_{X/S}\xrightarrow{+1}.
\]
Furthermore, for a fixed representative $K^\kdot_p$ of $\uOmega^p_X$, the first map is given by the natural factorization of the wedge map $K^\kdot_{p-1}\otimes f^*\Omega^1_S \to K^\kdot_{p}$. Therefore starting with the base case $\uOmega^{-1,+}_{X/S}\qis \uuOmega^{-1,+}_{X/S}=0$, we see that this agrees precisely with the defining triangle of $\uuOmega^{p,+}_{X/S}$ in \cite[Theorem 3.2(2)]{Nin26}, and consequently
\[
\uOmega^{p,+}_{X/S}\qis \uuOmega^{p,+}_{X/S}
\]
for all $p\ge 0$. A similar argument works for $\uOmega^{p, -}_{X/S}$, with the difference being the inductive cone construction. The distinguished triangle in \cite[2.4]{Kov97a} (see also \cite[\S 2]{Kov97}) reads, after index shift
\[
\frakF^{p, -}_{i+1} = \Cone(\frakF^{p, -}_{i}\otimes f^*\omega_S \to \frakF^{i, -}_{i}\otimes f^*\Omega^{i-p+1}_S)[-1] \qis \frakF^{p, -}_{i}\otimes f^*\omega_S.
\]
The last quasi-isomorphism holds for all $i< p-1$ since $\Omega^{i-p+1}_S=0$ for $i-p+1<0$. In other words, for all $0\le i\le p-1$,
\[
\frakF^{p,-}_{i}\qis \frakF^{p,-}_{0} \otimes f^*\omega_S^{\otimes i} \qis \uOmega^{n-p}_X\otimes f^*\omega_S^{\otimes i-1}.
\]
Now, the triangle for $i=p-1$ reads
\[
\frakF^{p,-}_p \to \frakF^{p,-}_{p-1}\otimes f^*\omega_S \to \frakF^{p-1, -}_{p-1} \xrightarrow{+1}.
\]
Using the definition $\uOmega^{n-p-1,-}_{X/S}=\frakF^{p,-}_p\otimes f^*\omega_S^{\otimes -p}$, this becomes
\[
\uOmega^{n-p-1,-}_{X/S} \otimes f^*\omega_S^{\otimes p} \to
\uOmega^{n-p}_X\otimes f^*\omega_S^{\otimes p-1} \to
\uOmega^{n-p,-}_{X/S} \otimes f^*\omega_S^{\otimes p-1} \xrightarrow{+1},
\]
which gives the defining triangle for $\uOmega^{p,-}_{X/S}$, cf. \cite[Theorem 1.3.2]{Kov96}. 
Note that after twisting with $f^*\omega_S^{p}$, the second map in the triangle factorizes the wedge map $K^{\kdot}_{n-p}\otimes f^*\omega_S\to K^{\kdot}_{n-p+1}$. Therefore, the same argument as the right complex concludes that
\[
\uOmega^{p,-}_{X/S}\qis \uuOmega^{p,-}_{X/S},
\]
once we identify the base case $\uOmega^{k,-}_{X/S}\qis \uuOmega^{k,-}_{X/S}=0$ for $k> \dim X-\dim f(X)$, as desired.
\end{proof}

\medskip
\section{Higher dimensional base change} \label{section base change}

The notions in \cite[Definition 4.1]{Nin26} extend naturally to our setting over a higher dimensional base.

\begin{definition}\label{def_bc}
Let $f: X\to S$ be a morphism of algebraic varieties with $\dim X=n$, $\dim S=k$ and $S$ smooth. Let $s\in S$, $m\in\bbZ$. We say the fiber $X_s$ satisfies
\begin{itemize}
\item \emph{left-$m$-base change} for the relative Du Bois complex if there is a natural quasi-isomorphism
\[
\uOmega^{p, -}_{X/S}\otimes^{\tL}_X\scrO_{X_s}\qis \uOmega^{p}_{X_s},
\]
for all $p\le m$.

\item \emph{right-$m$-base change} for the relative Du Bois complex if there is a natural quasi-isomorphism
\[
\uOmega^{p, +}_{X/S}\otimes^{\tL}_X\scrO_{X_s}\qis \uOmega^{p}_{X_s},
\]
for all $p\le m$.

\item \emph{$m$-base change} for the relative Du Bois complex if it satisfies both left- and right-$m$-base change.

\item \emph{base change} (resp. \emph{left-base change}, \emph{right-base change}) for the relative Du Bois complex if it satisfies $m$-base change (resp. left-$m$-base change, right-$m$-base change) for all $m\in \bbZ$.
\end{itemize}

We will frequently omit ``for the relative Du Bois complex'' if it is clear from the context.
\end{definition}

The following lemma identifies the relative complex over a point.

\begin{lemma}\label{lem_bc_dim0}
Let $f:X\to S$ be a morphism of algebraic varieties, with $S$ smooth of dimension $0$. Then for each $p\in \bbN$,
\[
\uOmega^{p,+}_{X/S}\qis \uOmega^{p,-}_{X/S} \qis \uOmega^p_X.
\]
\end{lemma}

\begin{proof}
Recall as in Construction \ref{cons_frakF} that we iteratively defined
\[
\frakF^{p,+}_{i+1} = \Cone (\frakF^{i,+}_{i}\otimes f^*\Omega^{p-i}_S\to \frakF^{p,+}_{i}).
\]
Since $\Omega^{p-i}_S=0$ for all $i\ne p$, we have
\[
\uOmega^p_X\qis \frakF^{p,+}_{0} \qis \dots \qis \frakF^{p,+}_{p} \qis \uOmega^{p,+}_{X/S}.
\]
For the left complex $\uOmega^{p,-}_{X/S}$, it can be defined similarly with a representative of $\uOmega^p_X$ as in Theorem \ref{thm_relDB} and then the proof would follow. However, since \cite{Kov97} defined it first assuming $X$ is smooth, we will check this case instead and use \cite[Lemma 3.3]{Kov97} to conclude that it holds for singular $X$ as well. When $X$ is smooth with $\dim X=n$, we have
\[
\frakF^{p,-}_0=\Omega^{n-p}_X\otimes f^*\omega_S^{-1}, 
\]
Iteratively, \cite[Proposition 2.3]{Kov97} defines
\[
\frakF^{p,-}_{i+1} = \Cone(\frakF^{p,-}_i\otimes f^*\omega_S \to \frakF^{i,-}_i\otimes f^*\Omega^{-p+i}_S)[-1],
\]
which then simplifies to
\[
\Omega^p_X \qis \frakF^{n-p,-}_{0} \qis \dots \qis \frakF^{n-p,-}_{n-p} \qis \uOmega^{p,-}_{X/S}.
\]
This concludes the proof for both complexes.
\end{proof}

\begin{lemma}\label{lem_hyperfilt_natural_map}
Let $f:X \to S$, $g:S\to T$ be morphisms of algebraic varieties, with $S, T$ smooth. Then there exists a functorial map 
\[
\uOmega^{p,+}_{X/T} \to \uOmega^{p,+}_{X/S}
\]
and a compatible map of hyperfiltrations $\bbF^{\kdot, +}_T \to \bbF^{\kdot, +}_S$. That is, for each $p, q\ge 0$, the diagram
\[\begin{tikzcd}[row sep=small]
\bbF^{q+1,+}_T \uOmega^p_X \ar[r] \ar[d]& 
\bbF^{q,+}_T \uOmega^p_X \ar[d] \\
\bbF^{q+1,+}_S \uOmega^p_X \ar[r]& 
\bbF^{q,+}_S \uOmega^p_X
\end{tikzcd}\]
commutes. In particular, they are also compatible with the maps from $\uOmega^p_X$.
\end{lemma}

\begin{proof}
Recall in Construction \ref{cons_frakF}, we iteratively defined intermediate objects $\frakF^{p,+}_{i,S}, \frakF^{p,+}_{i,T}$ for $p, i \ge 0$ associated to the morphisms $f:X\to S$, $g\circ f:X\to T$, with base case
\[
\frakF^{p,+}_{0,S} \qis \frakF^{p,+}_{0, T}\qis \uOmega^p_X.
\]
Here the additional subscripts distinguish the bases $S$ and $T$. The defining triangle for $i=0$ reads
\[\begin{tikzcd}
\uOmega^0_X\otimes (gf)^*\Omega^{p}_T \ar[r] \ar[d] &
\uOmega^p_X \ar[r] \ar[d, equal] &
\frakF^{p, +}_{1, T} \ar[r, "+1"] \ar[d, dotted] &
\, \\
\uOmega^0_X\otimes f^*\Omega^{p}_S \ar[r] &
\uOmega^p_X \ar[r] &
\frakF^{p, +}_{1, S} \ar[r, "+1"] &
\,
\end{tikzcd}\]
where the first map is induced by the natural inclusion of base forms $g^*\Omega^{p}_T\to \Omega^p_S$. At each step $i>0$, we may obtain comparison maps from the diagram of (\ref{box}),
\[\begin{tikzcd}
\frakF^{i, +}_{i, T}\otimes (gf)^*\Omega^{p-i}_T \ar[r] \ar[d] &
\frakF^{p, +}_{i, T} \ar[r] \ar[d] &
\frakF^{p, +}_{i+1, T} \ar[r, "+1"] \ar[d, dotted] &
\, \\
\frakF^{i, +}_{i, S}\otimes f^*\Omega^{p-i}_S \ar[r] &
\frakF^{p, +}_{i, S} \ar[r] &
\frakF^{p, +}_{i+1, S} \ar[r, "+1"] &
\,
\end{tikzcd}\]
using instead $g^*\Omega^{p-i}_T \to \Omega^{p-i}_S$. For $i+1=p$, this gives the desired $\uOmega^{p,+}_{X/T}\to \uOmega^{p,+}_{X/S}$. Now, we populate the maps for the hyperfiltrations analogously using Construction \ref{cons_hyperfilt}. We have a diagram
\[\begin{tikzcd}
\bbF^{q+1,+}_T \uOmega^p_X \ar[r] \ar[d, dotted] & 
\bbF^{q,+}_T \uOmega^p_X \ar[r] \ar[d] &
\uOmega^{p-q,+}_{X/T}\otimes (gf)^* \Omega^q_T \ar[r, "+1"] \ar[d] &
\, \\
\bbF^{q+1,+}_S \uOmega^p_X \ar[r] & 
\bbF^{q,+}_S \uOmega^p_X \ar[r] &
\uOmega^{p-q,+}_{X/S}\otimes f^* \Omega^q_S \ar[r, "+1"] &
\,
\end{tikzcd}\]
identifying the associated graded as in Theorem \ref{thm_relDB}(\ref{reldb2}). When $q=0$, the two hyperfiltration objects are naturally quasi-isomorphic
\[
\bbF^{0, +}_T \uOmega^p_X \qis \bbF^{0, +}_S \uOmega^p_X \qis \uOmega^p_X.
\]
Hence we may iteratively extend to $\bbF^{q+1}_T\to \bbF^{q+1}_S$ using the functorial map $\uOmega^{p-q,+}_{X/T} \to \uOmega^{p-q,+}_{X/S}$ defined previously, combined with the inclusion of base forms. This shows all the claimed properties.
\end{proof}

The following is a relative version of the triangle for the relative Du Bois complex over a curve.

\begin{proposition}\label{prop_rel_dim1_triangle}
Let $f: X\to S$, $g:S\to T$ be morphisms of algebraic varieties, with $S, T$ smooth and $g$ smooth of relative dimension $1$. Then for each $p\in \bbN$, there is a triangle
\[
\uOmega^{p-1,+}_{X/S}\otimes f^*\Omega^1_{S/T} \to \uOmega^{p,+}_{X/T}\to \uOmega^{p,+}_{X/S} \xrightarrow{+1}.
\]
\end{proposition}

\begin{proof}
This is an exercise of patience working with the hyperfiltrations associated with $\uOmega^{p,+}_{X/S}$ and $\uOmega^{p,+}_{X/T}$. We proceed first by induction on $p$, noting
\[
\uOmega^{-1, +}_{X/S}\qis 0 \quad \text{and} \quad \uOmega^0_X \qis \uOmega^{0,+}_{X/T} \qis \uOmega^{0,+}_{X/S}.
\]
Suppose the triangle in the statement holds up to some fixed $p\ge 0$. Using Lemma \ref{lem_hyperfilt_natural_map}, define
\[
C^{p, q} = \Cone(\bbF^{q,+}_T\uOmega^p_X \to \bbF^{q,+}_S\uOmega^p_X).
\]
Consider the following diagram associated to the hyperfiltrations in Theorem \ref{thm_relDB},
\begin{equation}\begin{tikzcd}\label{eqn_rel_dim1_diagram}
\bbF^{q+1,+}_T\uOmega^p_X \ar[r] \ar[d] &
\bbF^{q,+}_T\uOmega^p_X \ar[r] \ar[d] &
\uOmega^{p-q,+}_{X/T}\otimes (gf)^*\Omega^q_T \ar[d, "\alpha^{p,q}"] \ar[r, "+1"] &
\, \\
\bbF^{q+1,+}_S\uOmega^p_X \ar[r] \ar[d] &
\bbF^{q,+}_S\uOmega^p_X \ar[r] \ar[d] &
\uOmega^{p-q,+}_{X/S}\otimes f^*\Omega^q_S \ar[d] \ar[r, "+1"] &
\, \\
C^{p, q+1} \ar[r] \ar[d, "+1"]&
C^{p, q} \ar[r] \ar[d, "+1"]&
D^{p, q} \ar[r, "+1"] \ar[d, "+1"]&
\, \\
\, & \, & \, & \,
\end{tikzcd}\end{equation}
where using the derived nine-lemma there is an object $D^{p,q}$ making the rows and columns exact triangles. 

\begin{claim}\label{claim_rel_dim1_triangle}
For $q\ge 1$, we have
\[
C^{p,q} \qis \uOmega^{p-q, +}_{X/S}\otimes f^*\Omega^1_{S/T}\otimes (gf)^*\Omega^{q-1}_T.
\]
\end{claim}

Assuming Claim \ref{claim_rel_dim1_triangle}, the $q=0$ diagram of (\ref{eqn_rel_dim1_diagram}) reads
\[\begin{tikzcd}
\bbF^{1,+}_T\uOmega^p_X \ar[r] \ar[d] &
\uOmega^p_X \ar[r] \ar[d, equal] &
\uOmega^{p,+}_{X/T} \ar[d] \ar[r, "+1"] &
\, \\
\bbF^{1,+}_S\uOmega^p_X \ar[r] \ar[d] &
\uOmega^p_X \ar[r] \ar[d] &
\uOmega^{p,+}_{X/S} \ar[d] \ar[r, "+1"] &
\, \\
C^{p, 1} \ar[r] \ar[d, "+1"]&
0 \ar[r] \ar[d, "+1"]&
C^{p, 1}[1]. \ar[r, "+1"] \ar[d, "+1"]&
\, \\
\, & \, & \, & \,
\end{tikzcd}\]
We can simply read off the shifted third column to obtain the desired triangle
\[
C^{p, 1}\qis \uOmega^{p-1,+}_{X/S}\otimes f^*\Omega^1_{S/T} \to \uOmega^{p,+}_{X/T}\to \uOmega^{p,+}_{X/S} \xrightarrow{+1}
\]
This completes the proof of Proposition \ref{prop_rel_dim1_triangle}.
\end{proof}

\begin{proof}[Proof of Claim \ref{claim_rel_dim1_triangle}]
We work with the same inductive assumptions as in the proof of Proposition \ref{prop_rel_dim1_triangle}, i.e. assuming the proposition holds up to some $p\ge 0$. Now, the claim is local on $X$ by Theorem \ref{thm_relDB}(\ref{reldb5}). Note that $g$ is smooth of relative dimension $1$, so after shrinking the base if necessary, we make use of the local splitting
\[
\Omega^q_S\isom g^*\Omega^q_T \oplus (\Omega^1_{S/T}\otimes g^*\Omega^{q-1}_T).
\]
Thus for $q\ge 1$, we may apply the inductive hypothesis to the third column of (\ref{eqn_rel_dim1_diagram}), and obtain
\[
D^{p,q}\qis \left( \uOmega^{p-q-1,+}_{X/S}\otimes  f^*\Omega^1_{S/T}\otimes (gf)^*\Omega^{q}_T\right)[1]\oplus \left(\uOmega^{p-q, +}_{X/S}\otimes f^*\Omega^1_{S/T}\otimes (gf)^*\Omega^{q-1}_T\right)[0].
\]
Note that as in Theorem \ref{thm_relDB}, $\bbF^{q,+}_S\uOmega^p_X = 0$ for $q>\dim S$ and $\bbF^{q,+}_T\uOmega^p_X = 0$ for $q>\dim T = \dim S-1$. Therefore when $q=\dim S$, (\ref{eqn_rel_dim1_diagram}) consists of all zeroes except in the lower right square, and yields
\[
C^{p,q}\qis D^{p,q} \qis \uOmega^{p-q,+}_{X/S}\otimes f^*\Omega^q_S \qis \uOmega^{p-q,+}_{X/S}\otimes f^* \Omega^1_{S/T}\otimes (gf)^*\Omega^{q-1}_T.
\]
Here we used the splitting of $\Omega^q_S$ and that $\Omega^q_T=0$ for $q=\dim S$.

Now we proceed by descending induction on $q$ from $q=\dim S$, with the base case just shown. Note that as in Lemma \ref{lem_hyperfilt_natural_map}, the map $\alpha^{p,q}$ in (\ref{eqn_rel_dim1_diagram}) factors as
\[
\uOmega^{p-q,+}_{X/T}\otimes (gf)^*\Omega^q_T \xrightarrow{\beta^{p,q}} \uOmega^{p-q,+}_{X/S}\otimes (gf)^*\Omega^q_T \to \uOmega^{p-q,+}_{X/S}\otimes f^*\Omega^q_S
\]
with the last map induced by a split injection $(gf)^*\Omega^q_T\into f^*\Omega^q_S$. Using both inductive hypotheses on $p$ and $q$ for the source and the target, the composition realizes a quasi-isomorphism
\[
\uOmega^{p-q-1,+}_{X/S}\otimes f^*\Omega^1_{S/T} \otimes (gf)^*\Omega^{q}_T [1] \qis
\Cone(\beta^{p,q}) \to
D^{p,q} \qis \Cone(\alpha^{p,q}) \to
C^{p,q+1}[1].
\]
Furthermore, the first map is split injective induced by the splitting $(gf)^*\Omega^q_T\into f^*\Omega^q_S$. More precisely, it comes from this diagram extending left inverses:
\[\begin{tikzcd}[row sep=scriptsize]
\uOmega^{p-q,+}_{X/T} \otimes (gf)^* \Omega^q_T \ar[r, equal] \ar[d, "\beta^{p,q}"] &
\uOmega^{p-q,+}_{X/T} \otimes (gf)^* \Omega^q_T \ar[d, "\alpha^{p,q}"] \\
\uOmega^{p-q,+}_{X/S} \otimes (gf)^* \Omega^q_T \ar[r] \ar[d] &
\uOmega^{p-q,+}_{X/S} \otimes f^* \Omega^q_S \ar[d] \ar[l, bend right=15]\\
\Cone(\beta^{p,q}) \ar[r] \ar[d, "+1"]&
\Cone(\alpha^{p,q}). \ar[l, bend right=15, dotted] \ar[d, "+1"] \\
\, & \,
\end{tikzcd}\]
Combining the above, we see that the triangle coming from (\ref{eqn_rel_dim1_diagram})
\[
C^{p,q}\to D^{p,q} \to C^{p,q+1}[1] \xrightarrow{+1}
\]
is split. In particular, again applying the inductive hypothesis on $q$ identifying $C^{p,q+1}$ shows that
\[
C^{p,q} \qis \uOmega^{p-q, +}_{X/S}\otimes f^*\Omega^1_{S/T}\otimes (gf)^*\Omega^{q-1}_T,
\]
thus proving the claim.
\end{proof}

\begin{theorem} \label{base change}
Let $f: X\to S$ be a flat morphism of algebraic varieties with $\dim X=n$, $\dim S=k$ and $S$ smooth. Let $s\in S$ be a closed point and suppose the fiber $X_s$ has strict-$m$-Du Bois singularities. Then $X_s$ satisfies right-$m$-base change for the relative Du Bois complex, that is, for all $0\le p\le m$,
\[
\uOmega^{p, +}_{X/S}\otimes^{\tL}_X\scrO_{X_s} \qis \uOmega^p_{X_s}.
\]
\end{theorem}

\begin{proof}
The case $\dim S=0$ is Lemma \ref{lem_bc_dim0} and the case $\dim S=1$ follows from Proposition \ref{prop_curve_compare} and \cite[Theorem 1.1(2)]{Nin26}. We will prove this by induction on the dimension of the base, so suppose $\dim S=k\ge 2$ and the theorem has been proved for all bases of $\dim \le k-1$.

The statement is local on the base, so after possibly shrinking $S$, choose \'etale coordinates $t_1,\dots, t_k$ and project from $t_k$ to obtain $g: S\to T=\bbA^{k-1}_{t_1,\dots, t_{k-1}}$. We see that $g$ is smooth of relative dimension $1$ and $g\circ f$ is flat.

Fix $s\in S$ in the theorem and let $t=g(s)$, so that we have a family $X_t \to S_t$ over a smooth curve $S_t$. Over $s\in S_t$, we have an inclusion of an effective Cartier divisor $X_s\to X_t$ by flatness of $f$. By assumption, $X_s$ is strict-$m$-Du Bois. In particular note that $X_t$ is reduced in a neighborhood of $X_s$. Therefore, after further shrinking of $X$ to a neighborhood containing $X_s$, \cite[Theorem 1.3]{Nin26} implies that $X_t$ is strict-$m$-Du Bois. We have thus arrived at the scenario of the diagram
\[\begin{tikzcd}[row sep=scriptsize]
X_s \ar[r, hookrightarrow] \ar[d] &
X_t \ar[r, hookrightarrow] \ar[d, "f_t"] &
X \ar[d, "f"] \\
\{s\} \ar[r, hookrightarrow] &
S_t \ar[r, hookrightarrow] \ar[d] &
S \ar[d, "g"] \\
&
\{t\} \ar[r, hookrightarrow] &
T
\end{tikzcd}\]
where each square is Cartesian. Now, Proposition \ref{prop_rel_dim1_triangle} pulled back to $X_t$ gives a triangle
\[
\uOmega^{p-1,+}_{X/S}\otimes^{\tL}_X\scrO_{X_t} \otimes f_t^*\Omega^1_{S_t} \to 
\uOmega^{p,+}_{X/T} \otimes^{\tL}_X\scrO_{X_t} \to 
\uOmega^{p,+}_{X/S} \otimes^{\tL}_X\scrO_{X_t} \xrightarrow{+1},
\]
where we used that $f^*\Omega^1_{S/T}\otimes \scrO_{X_t}\isom f_t^* \Omega^1_{S_t/\kappa(t)}$. By functoriality, we have a diagram
\[\begin{tikzcd}
\uOmega^{p-1,+}_{X/S}\otimes^{\tL}_X\scrO_{X_t} \otimes f_t^*\Omega^1_{S_t} \ar[r] \ar[d, "\qis"] &
\uOmega^{p,+}_{X/T} \otimes^{\tL}_X\scrO_{X_t} \ar[r] \ar[d, "\qis"] &
\uOmega^{p,+}_{X/S} \otimes^{\tL}_X\scrO_{X_t} \ar[r, "+1"] \ar[d] &
\, \\
\uOmega^{p-1,+}_{X_t/S_t} \otimes f_t^*\Omega^1_{S_t} \ar[r] &
\uOmega^{p}_{X_t} \ar[r] &
\uOmega^{p,+}_{X_t/S_t} \ar[r, "+1"] &
.\,
\end{tikzcd}\]
Applying the inductive hypothesis to $g\circ f$ gives base change of the middle vertical map
\[
\uOmega^{p,+}_{X/T}\otimes^{\tL}_X\scrO_{X_t} \qis \uOmega^p_{X_t}
\]
for $0\le p\le m$. On the other hand, we may perform induction on $p\ge 0$, starting from 
\[
\uOmega^{-1,+}_{X/S}\qis \uOmega^{-1,+}_{X_t/S_t}\qis 0,
\]
and see that the left vertical map is also a quasi-isomorphism. Therefore by the five-lemma we have
\[
\uOmega^{p,+}_{X/S} \otimes^{\tL}_X\scrO_{X_t} \qis \uOmega^{p,+}_{X_t/S_t}.
\]
Finally, apply curve base change \cite[Theorem 1.1(2)]{Nin26} to the family $X_t\to S_t$ to obtain
\[
\uOmega^{p,+}_{X/S} \otimes^{\tL}_X\scrO_{X_s} \qis \uOmega^{p,+}_{X_t/S_t} \otimes^{\tL}_{X_t}\scrO_{X_s} \qis \uOmega^p_{X_s}
\]
for $0\le p\le m$, as desired.
\end{proof}

\medskip
\bibliographystyle{alpha} 
\bibliography{ref}

\end{document}